\documentclass{article} 

\usepackage{amsfonts}
\usepackage{amsmath}
\usepackage{amsthm} 
\usepackage{amssymb}

\usepackage{graphicx}    

\usepackage{algorithm}
\usepackage{algorithmic}

\usepackage{fullpage} 

\newcommand{\bbE}{\mathbb{E}}
\newcommand{\argmax}{\textrm{argmax}}

\newcommand{\Om}{\Omega}

\newtheorem{thm}{Theorem}
\newtheorem{lem}[thm]{Lemma}
\newtheorem{cor}[thm]{Corollary}

\renewcommand{\algorithmicensure}{\textbf{output:}}

\newcommand{\MAIN}[1]{ \renewcommand{\algorithmicensure}{\textbf{main:}} \ENSURE #1 \renewcommand{\algorithmicensure}{\textbf{output:}}}

\title{Optimal State-Space Reduction for Exact Calculation on Pedigree Hidden Markov Models\thanks{This work was partially supported by NSF grants OISE-0730136 and DMS-1106770.}}

\author{B. Kirkpatrick\footnotemark[2], 
	K. Kirkpatrick\footnotemark[3]}

\begin{document}         

\maketitle
\footnotetext[2]{Computer Science, University of British Columbia ({\tt bbkirk@cs.ubc.ca})} 
\footnotetext[3]{Mathematics, University of Illinois at Urbana-Champaign ({\tt kkirkpat@illinois.edu}) }

\begin{abstract}
To analyze whole-genome genetic data inherited in families, the likelihood is typically obtained from a Hidden Markov Model (HMM) having a state space of $2^n$ hidden states where $n$ is the number of meioses or edges in the pedigree.  There have been several attempts to speed up this calculation by reducing the state-space of the HMM.  One of these methods has been automated in a calculation that is more efficient than the na\"{i}ve HMM calculation; however, that method treats a special case and the efficiency gain is available for only those rare pedigrees containing long chains of single-child lineages. The other existing state-space reduction method treats the general case, but the existing algorithm has super-exponential running time.

We present three formulations of the state-space reduction problem, two dealing with groups and one with partitions. One of these problems, the maximum isometry group problem was discussed in detail by Browning and Browning~\cite{Browning2002}. We show that for pedigrees, all three of these problems have identical solutions. Furthermore, we are able to prove the uniqueness of the solution using the algorithm that we introduce. This algorithm leverages the insight provided by the equivalence between the partition and group formulations of the problem to quickly find the optimal state-space reduction for general pedigrees.

We propose a new likelihood calculation which is a two-stage process: find the optimal state-space, then run the HMM forward-backward algorithm on the optimal state-space. In comparison with the one-stage HMM calculation, this new method more quickly calculates the exact pedigree likelihood.

\end{abstract}



\section{Introduction}

\paragraph{Motivation}
Statistical calculations on pedigrees are the principal method behind
the most accurate disease-association
approaches~\cite{Risch1996,Thornton2007}.  In those approaches, the
aim is to find the regions of the genome that are associated with the
presence or absence of a disease among related individuals.
Furthermore, pedigree likelihoods are used to estimate fine-scale
recombination rates in humans~\cite{Coop2008}, where there are few
other approaches for making these estimates.  There exist many implementations of exact likelihood calculations for pedigrees~\cite{Fishelson2005,Abecasis2002,Sobel1996}.  Computation of
probabilities on pedigrees are of great interest to computer
scientists because they give an important example of graphical models
which model probability distributions by using a graph whose edges are
conditional probability events and whose nodes are random
variables~\cite{Lauritzen2003}.  Methods for reducing the state-space
of a pedigree graphical model could generalize to other graphical
models, as noted also by Geiger et al~\cite{Geiger2009}.

\paragraph{The Problem Summary}
Hidden Markov Models (HMMs) analyzing the genotypes of related
individuals have running time $O(m2^{2n})$ where $m$ is the number of
sites and $n$ is the number of meioses in the pedigree.  Therefore, it
is desirable to find more efficient algorithms. Any partitioning of
the state space into $k$ ensemble states (i.e., states with identical
emission probabilities and Markovian transition probabilities) will
improve the running time of an HMM to $O(mk^2)$, even if the ensembles
are not optimal.  Since the HMMs have an exponential state space and a
running time polynomial in the size of the state space, even an
exponential algorithm for finding ensemble states can improve the
running time of the HMM calculations.

\paragraph{Literature Review}
Donnelly~\cite{Donnelly1983} introduced the idea of finding ensemble
states for the IBD Markov model, and used a manual method for finding
the symmetries for several examples of two-person pedigrees. Browning
and Browning~\cite{Browning2002} formalized the requirements for
symmetries that describe ensemble states in a new HMM.  They gave the
first algorithm for finding the maximal set of isometries that
preserves the Markov property and the IBD information. Their algorithm
which is based on enumerating permutations appears to have worst-case
running time of at least of $O(n!2^{2n})$, where $n$ is the number of
meioses in the pedigree.  However, the running time of their algorithm
is difficult to analyze due to their three case-specific improvements. They also left open the question of whether
groups other than isometry groups could give useful state-space
reductions~\cite{Browning2002}.  
Browning and Browning found the maximal \emph{group of isometries} satisfying the constraints, however, they did not
draw any conclusions about whether their method finds the \emph{group} with
the maximal orbit sizes.

McPeek~\cite{McPeek2002inference} presented a detailed formulation of
the condensed identity states and an algorithm.  Most recently Geiger
et al~\cite{Geiger2009} discussed a similar problem using the language
of partitions.  They found isometries of a limited type in $O(n^2)$.  
They gave a
special-case state-space reduction involving only partitions that
collapse simple lineages (multiple generations with a single child per
generation and with the non-lineage parents being founders).
Several other people have introduced algorithms for finding symmetries for
systems applications~\cite{Lorentson2001,Junttila2004}.

Kirkpatrick~\cite{Kirkpatrick2011xxxx} used a method of finding the state space which is the maximal group of isometries (i.e.~such the method in Browning and Browning~\cite{Browning2002} or in this paper) to determine whether two pedigrees are non-identifiable, meaning that under any fixed data the two pedigrees have the same probability of generating the data.  This is important in the context of pedigree reconstruction where the problem is to infer a pedigree graph only from genetic data.  The reconstruction algorithm is typically viewed as a maximum-likelihood search over pedigree graphs where each pedigree is scored using the likelihood.  Non-identifiability, which is computed using a method such as the one in this paper, says that the correct pedigree graph cannot be inferred with high probability because of ties in the likelihood score.

\paragraph{Our Contribution}
Inspired by the work of Browning and Browning~\cite{Browning2002}, we
look for maximal ensembles of the hidden states that can be used to
create a new HMM with a much more efficient running-time. We
introduce an improved algorithm for finding the maximal ensemble
states that preserve both the Markov property 
and the identity by descent (IBD) information of the individuals of
interest. 

We introduce an $O(n2^{2n})$ maximal-ensemble algorithm for finding a
permutation group on the $2^n$ vertices of the hypercube, and for
producing the most efficient ensemble states (i.e. the smallest
partition of the state-space that respects the IBD and Markov
properties and has the maximal partition sets and minimal number of
sets in the partition).  We prove that the optimal partition is a
solution to the maximal isometry group problem that Browning and
Browning introduced, thereby relating the work of Geiger, et al to
that of Browning and Browning.  Both Browning and Browning's algorithm
and ours finds the optimal partition of the state space which can be
described using a group of isometries having a maximal number of
elements.  However, our algorithm is much faster, having a coefficient
$n$ instead of $n!$.

We also introduce a bootstrap version of the maximal-ensemble
algorithm which takes advantage of the isometries introduced by
Geiger, et al.~\cite{Geiger2009} and the well-known founder isometry.
By enumerating one representative from each set of the partition
induced by the known isometries, we can create a bootstrap
maximal-ensemble algorithm that runs in $O(nk2^{n})$ time where $n$ is
the number of meioses in the pedigree, and $k$ is the number of
partitions from the known isometries.

\section{Problem Description}
Consider a pedigree graph, $P$, having individuals $V$ as nodes and
having $n$ meioses with each meiosis being a directed edge from parent
to child.  Let $I$ be the set of individuals of interest, 
because we have data for those individuals.  While it might be
algorithmically convenient to assume that $I = V$, it is impractical.
Many of the ancestral individuals in the pedigree are likely deceased,
and genetic samples are unavailable.

An \emph{inheritance
state or vector} is a binary vector $x$ with $n$ bits where each bit
indicates which grand-parental allele, paternal or maternal, was
copied for that meiosis.  
The equivalent
\emph{inheritance graph}, $R_x$, has two nodes per individual (one for
each allele) and edges from inherited parental alleles to their
corresponding child alleles.  
Individuals of interest are called
\emph{identical by descent (IBD)} if a particular founder allele was
copied to each of the individuals.  In general, the inheritance graph
is a collection of trees, since each allele is copied from a single parent.

The set of all inheritance states (binary $n$-vectors) is the $n$-dimensional 
hypercube $\mathcal{H}_{n}$, with $2^n$ vertices.  The inheritance
process is modelled as a symmetric random walk on $\mathcal{H}_{n}$, with
the time dimension of the walk being the distance along the genome.  At
equilibrium, the walk has uniform probability of being at any of the
hypercube vertices.  From vertex $x$ in $\mathcal{H}_n$, a step is
taken to a neighboring vertex after an exponential waiting time with
parameter $\lambda=n$.
 For each individual zygote, with one meiosis,
this is a Poisson process with parameter $\lambda=1$ and genome length
roughly $30$.

There is a discrete version of this random walk, which is often
used for hidden Markov models (HMMs) that compute the probability of
observing the given data by taking an expectation over the possible
random walks on the hypercube.  Let $X$ be a Markov chain, $\{X_t:
t=1,2,...,m\}$ for $m$ loci with a state space $\mathcal{H}_{n}$
consisting of all the inheritance states of the pedigree.  The
recombination rate, $\theta_t$, is the probability of recombination
per meiosis, between a neighboring pair of loci, $t$ and $t+1$.  If
$t$ and $t+1$ are separated by distance $d$, then the Poisson process
tells us that the probability of an odd number of recombinations is
$\theta_t = 1/2(1-e^{-2\lambda d})$.  The natural distance on $\mathcal{H}_{n}$ is the Hamming distance, $|x \oplus y|$, for two states $x$ and $y$, where $\oplus$ is the XOR operation and $|\cdot|$ is the $L^1$-norm in $\mathbb{R}^n$.
Then the probability of transitioning from $x$ to $y$ is
\[Pr[X_{t+1}=y|X_t=x] = \theta_t^{|x \oplus y|}(1-\theta_t)^{n-|x \oplus y|}. \]
Figure~\ref{fig:halfsibs} shows an example HMM with three genomic sites.  The states of the HMM are shown in circles on the right.

\begin{figure}[ht]
  \begin{center}
    \includegraphics[scale=0.6]{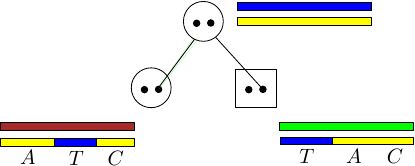}
    \hspace{1cm}
    \includegraphics[scale=0.6]{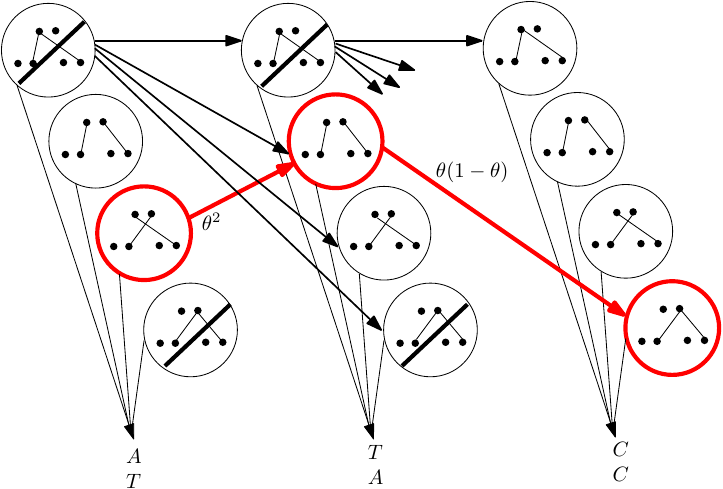}
  \end{center}
  \caption{{\bf Two Half-Siblings.}   \label{fig:halfsibs}
(Left Panel) A pedigree with two non-founders of which two are
half-siblings together with their common parent.  Circles and boxes
represent female and male individuals, respectively, while the two black dots for
each person represent their two chromosomes or alleles.  Edges are
implicitly directed downward from parent to child.  The alleles of
each individual are ordered, so that the left allele, or paternal
allele, is inherited from the person's father, while the right,
maternal allele is inherited from the mother.  The two siblings are
the only labeled individuals.  Their genomes are shown in color so
that the same color indicates inheritance from the same ancestor.  For
convenience, the genotype of each person is homozygous.
(Right Panel) The HMM for the genotypes from the left panel.  At each
site in the genome, the possibles states are the vectors in ${\cal
H}_n$.  In each circle an inheritance state is drawn as an inheritance
graph and the inheritance states for a single site are arranged in a
column.  The allowed transitions between neighboring sites are a
complete bipartite graph (due to space, only a fraction of the edges
are drawn).  The nodes with a slash through them are inheritance
states that are not allowed by the data.  The red nodes and edges are
the path for the actual inheritance states indicated by the yellow and
blue in the left panel.  However, this is only one of several paths of
inheritance states that are consistent with the data.  
}
\end{figure}

We define potential ensembles of states as being the orbits of a group
of symmetries.  Let $G$ be a group that acts on the state space
$\mathcal{H}_{n}$ of $X$.  A symmetry is a bijection $\psi \in G$
where $\psi$ is a permutation on $2^n$ elements, the vertices of
$\mathcal{H}_{n}$.  An \emph{orbit} of $G$ acting on
$\mathcal{H}_{n}$ is the set
\[\omega(y) = \{x | x = \psi(y) ~\textrm{and}~ \psi \in G\},\]
and we write the set of all orbits of $G$ as $\Om(G) = \{\omega(y): y \in \mathcal{H}_{n}\}$.

Conventional algorithms for computing likelihoods of data have an
exponential running time, because the state space of the HMM is
exponential in the number of meioses in the pedigree. We propose new
ways to collapse hypercube vertices into ensemble states for a new HMM
that has a more efficient running time.  In particular we are
interested in optimal ensemble states that preserve certain
relationship structures: the Markovianness of the random walk and
the emission probabilities.  We will first discuss the Markov property
and then discuss the constraints on ensemble states that the emission
probabilities provide.

\subsection{Markov Property}

Let $\{X_t\}$ be a stationary, reversible Markov chain with state
space $\Omega$, such as the chain corresponding to the hidden states
of the pedigree HMM.

Let $Y$ be a new processes, $\{Y_t: t=1,2,...,m\}$ having
states $\Omega(G) = \{\omega_1,...,\omega_k\}$ which are the orbits of some group $G$.  This new
Markov chain is coupled to the original such that when $X_t=x \in
\omega \in \Omega(G)$, $Y_t = \omega$, and $Y_t$ is a projection of $X_t$ 
into a smaller state space.  Define the transition probabilities for
process $Y_t$ as
\begin{eqnarray}
\label{eqn:transitions}
Pr[Y_{t+1} = \omega_j | Y_t= \omega_i] = \sum_{y \in \omega_j} Pr[X_{t+1}=y|X_t=x] 
\end{eqnarray}
for $x \in \omega_i$, for $\omega_i,\omega_j \in \Omega(G)$.
We will call $Y_t$ the expectation chain since 
\begin{eqnarray*} 
Pr[Y_{t+1}=\omega_j | Y_t=\omega_i] &=& \bbE[E_j | X_t=x],
\end{eqnarray*}
where $E_j$ is the event that $X_{t+1} \in \omega_j$.

Since $X_t$ is stationary and reversible, the necessary and sufficient
condition~\cite{Burke1958} for $Y_t$ to also be Markov is that
\begin{equation}
\label{eqn:markov}
\sum_{y \in \omega_j} Pr[X_{t+1}=y|X_t=x_1] = \sum_{y \in \omega_j}  Pr[X_{t+1}=y|X_t=x_2]
\end{equation}
for all $x_1,x_2 \in \omega_i$ for all $i$, and for all $\omega_j$.
Therefore any group whose orbits satisfy this set of equations can be
used to create a new Markov chain $Y_t$.

From Equations~(\ref{eqn:transitions}) and~(\ref{eqn:markov}), we see
that the stationary distribution of Markov chain $Y_t$ is $Pr[Y_t =
\omega_i] = \sum_{y \in \omega_i} \pi_y$ where $\pi_y$ is the stationary
distribution of $X_t$.  For pedigree HMMs, the stationary distribution
of $X_t$ is uniform, $\pi_y = 1/2^n$, therefore the expectation chain
for some group that satisfies Equation~(\ref{eqn:markov}) will have a
stationary distribution $Pr[Y_t = \omega_i] = |\omega_i| / 2^n$.

For pedigree Markov chains, Equation~(\ref{eqn:markov}) becomes, for $s = \theta/(1-\theta)$ and $0 < \theta < 0.5$,
\begin{equation}
\label{eqn:markovped}
\sum_{y \in \omega_j}  s^{|y \oplus x_1|} = \sum_{y \in \omega_j} s^{|y \oplus x_2|} ~~~\forall x_1,x_2 \in \omega_i.
\end{equation} 
If the expectation chain $Y_t$ corresponding to pedigree Markov chain $X_t$
satisfies this equation, we say that it satisfies the \emph{Markov
property}.  Notice that these polynomials are identical if and only if
the coefficients of like powers are equal.

Browning and Browning~\cite{Browning2002} gave an algorithm that
searches for a maximal group of isometries where the group was maximal
in the number of group elements.  A group, $G$, of \emph{isometries} has orbits
$\Omega(G) = \{\omega_1,...,\omega_k\}$ such that $|T(x) \oplus T(y)| = |x \oplus y|$ for all $T \in G$,
$y \in \omega_j$ and $x \in \omega_i$ for all $i,j$.  We will refer to isometries using $T$ and will reserve $\psi$ for general symmetries.

This means that the
transition probabilities are related by
\begin{equation}
Pr[X_{t+1}=y | X_t = x] = Pr[X_{t+1}=T(y) | X_t=T(x)].
\end{equation}
Browning and Browning left open the question of whether any symmetry
groups satisfying Equation~(\ref{eqn:markovped}) must be equivalent to
a group of isometries (meaning that it has the same orbits).  We
answer this question. Theorem~\ref{thm:isometry} proves that for any
group of permutations satisfying Equation~\ref{eqn:markovped}, there
is always a group of isometries having the same orbits as the group of
permutations.

\begin{thm}
\label{thm:isometry}
Let $S$ be a group of permutations of $\mathcal{H}_{n}$ whose orbits $\Om(S)$ satisfy Equation~(\ref{eqn:markovped}). Then there exists a  group of isometries $G$ having the same orbits as $S$: that is, for every $T \in G$ and all $x,y \in \mathcal{H}_{n}$, $|y \oplus x| = |T(y) \oplus T(x)|$, and the set of orbits $\Om(G)$ is equal to $\Om(S)$.
\end{thm}
\begin{proof}
We prove this by constructing a generating set $A$ for $G$.  
First, let the identity permutation $\pi_e$ be in $A$.  
Then for each orbit $\omega$ of $S$, and each pair of points $x_1$ and $x_2$ in
$\omega$, we will construct a permutation $\pi_{x_1,x_2}$ to 
add to the generating set $A$.  If $x_1=x_2$, then $\pi_{x_1,x_2} =
\pi_e$ which is already in $A$.  If $x_1 \ne x_2$ then $\pi_{x_1,x_2}$ 
will be a composition of disjoint two-cycles, in particular including the cycle $(x_1~ x_2)$.  
Our generating set
$A$ will then be the union of all these permutations, so by construction it 
will generate a group $G = \langle A \rangle$ having the same orbits as $S$.

For fixed $x_1,x_2\in \omega$, the two-cycles comprising $\pi_{x_1,x_2}$ are constructed as follows: 

For each $k = 1, \dots, n$, define $a_k := \# \{y \in \omega: |y \oplus x_1|=k \}$ and $b_k := \# \{z \in \omega: |z \oplus x_2|=k \}$, which implies by Equation~(\ref{eqn:markovped}) that $a_k s^k = b_k s^k$ for each $k$, and hence $a_k = b_k$,  since $s > 0$ and polynomials in $s$ are uniquely determined by their coefficients and powers.  Then, for each given $y_1 \in \omega$ that is distinct from both $x_1$ and $x_2$, there exists $z_1$ such that $|y_1 \oplus x_1| = |z_1 \oplus x_2| =k$, because $a_k \ge 1$, a consequence of the fact that $y_1 \in A_k := \{y \in \omega: |y \oplus x_1|=k \}$. In other words, $z_1:= y_1 \oplus (x_1 \oplus x_2)$, and the cycle is $c_1:=(y_1~z_1)$. 

Proceed similarly for $y_2 \in {\cal H}_n \setminus \{x_1, x_2, y_1, z_1\}$, defining $z_2:= y_2 \oplus (x_1 \oplus x_2)$, and the cycle $c_2:=(y_2~z_2)$, and so on for each $y_i \in {\cal H}_n \setminus \{x_1, x_2, y_1, z_1, \dots, y_{i-1}, z_{i-1}\}$, with $z_i:= y_i \oplus (x_1 \oplus x_2)$ and $c_i:=(y_i~z_i)$. Then we define the permutation $\pi_{x_1,x_2} := c_1 \circ c_2 \circ ... \circ c_{2^{n}}$. In particular it has the cycle $(x_1~ x_2)$ in its composition, since when $y=x_1$, we have $z=x_2$. Notice also that the definitions of $z_i$ imply that 
\begin{align}
y_i \oplus y_j =  y_i \oplus y_1 \oplus y_j  \oplus y_1 =z_i \oplus z_1 \oplus z_j  \oplus z_1 = z_i \oplus z_j; \\
y_i \oplus z_j =  y_i \oplus x_1 \oplus z_j  \oplus x_1 =z_i \oplus x_2 \oplus y_j  \oplus x_2 = z_i \oplus y_j.
\end{align}
Hence by taking $L^1$ norms, the permutation $\pi_{x_1,x_2}$ is an isometry with respect to Hamming distance.



Furthermore, the group $G = \langle A \rangle$ will have the same
orbits as $S$, since for each orbit $\omega$ and each pair
$x_1,x_2\in \omega$, the cycle $(x_1~ x_2)$ will appear in some permutation, and no pair
of points from different orbits will appear as a cycle in any permutation.
\end{proof}

This proof complements the result from Browning and Browning regarding the fact that isometry groups always satisfy Equation~\ref{eqn:markovped}.  Indeed, we will state the complete result as a corollary.

\begin{cor} A group $S$ has orbits $\Omega(S)$ satisfying Equation~\ref{eqn:markovped} if and only if there is an isometry group $G$ whose orbits $\Omega(G)$ are identical to $\Omega(S)$.
\end{cor}
\begin{proof}
Browning and Browning~\cite{Browning2002} showed that all isometry groups $G$ satisfy Equation~\ref{eqn:markovped}.  Theorem~\ref{thm:isometry} completes the proof.
\end{proof}

It is a well-known fact in algebra that any partition can be the orbits of some symmetry group, and that the orbits of any symmetry group are a partition~\cite{Durbin2000}.  We will recapitulate this simple result next.

\begin{cor}
\label{cor:partition}
A partition satisfies Equation~\ref{eqn:markovped} if and only if it is equivalent to the orbits of some isometry group.
\end{cor}
\begin{proof}
Assume we are given a partition $\{W_1,...,W_k\}$ of set ${\cal H}_n$
where $W_i \cup W_j = \emptyset$, $\cup_{i} W_i = {\cal H}_n$ and the
partition satisfies Equation~\ref{eqn:markovped}. We will create a
symmetry group $S$ whose orbits $\Omega(S) = \{W_1,...,W_k\}$.  This
is easily done.  For each set in the partition $W_i$, create a
permutation with a single cycle $\pi_i = (y_1~y_2~...~y_l)$ where all
$y_j \in W_i$.  Make a generating set $A = \{\pi_i : 1 \le i \le k\}
\cup \pi_e$ where $\pi_e$ is the identity permutation.  Then group $S
= \langle A \rangle$ clearly has orbits $\Omega(s) = \{W_1,...,W_k\}$.
By Theorem~\ref{thm:isometry} there is an isometry group with the same
orbits.

Assume we are given an isometry group $G$.  Clearly, by Browning and
Browning's proof~\cite{Browning2002}, the orbits define a partition
$\Omega(G)$ that satisfies Equation~\ref{eqn:markovped}.
\end{proof}

Browning and Browning~\cite{Browning2002}  also showed that any isometry
$T:\mathcal{H}_{n} \to \mathcal{H}_{n}$ can be uniquely written as
$T=\pi \circ \phi_a$ where $\pi$ is a permutation on $n$ elements, the
bits of the hypercube vertex, and $\phi_a$ is a switch function where
$\phi_a(x) = a \oplus x$ where $\oplus$ is the bit-wise XOR operation.

An isometry describes some aspect of the pedigree graph.  For example,
an isometry consisting of a switch and the identity permutation can be
used to enumerate one element from each orbit by simply fixing the
1-bit's value and then enumerating all possible values for the other
switch bits.  On the other hand, an isometry consisting of the
identity switch (all zero) and a permutation of one cycle can be used
to enumerate one element for each orbit by listing the 1-prefixes of
the permuted bits, (i.e.~for three bits, the representatives are
$000$, $100$, $110$, and $111$).

\subsection{Emission Property}

The Markov property is not enough to ensure that the HMM based on
Markov chain $Y_t$ has the same likelihood as the HMM for $X_t$.  In
order to ensure this, we introduce a property on the emission
probabilities, namely that all the elements in one orbit must have
identical emission probabilities.  We call these orbits the emission 
partition, since they are induced by the emission probability.
In order to define this object, we need to introduce some more notation.

Recall that $R_x$ is the inheritance graph for inheritance vector $x$.
The relationship structures we wish to preserve are the IBD
relationships on the individuals of interest $I$.  Relationships on individuals translate to relationships between their alleles.  Let $I_m$ be the
maternal alleles of all the individuals of interest and $I_p$ be the
paternal alleles of all the individuals of interest.  The inheritance
graph $R_x$ is a forest; let $CC(R_x)$ refer to the connected
components of $R_x$ which are labeled with $I_m \cup I_p$.
The same-labeled connected components induce a partition 
\[ D_x = \{y \in \mathcal{H}_{n} | CC(R_y) = CC(R_x)\}.\]
We call the partition $D := \{D_x | \forall x\}$ the identity states, since it indicates a particular 
identity-by-descent (IBD) relationship among the labeled individuals.
These have been well studied~\cite{Jacquard1972,Thompson1974,Karigl1982}.

Looking at a small example, containing two siblings who are the
individuals of interest and their two parents, we see that the
identity states are:
\begin{eqnarray*}
D_{0000} &=& \{0000, 0101,1010,1111\}, \\
D_{1000} &=& \{1000,0010,1101,0111\}, \\
D_{0100} &=& \{0100,1110,0001,1011\}, \\
D_{1100} &=&\{1100,0110,1001,0011\},
\end{eqnarray*}
where the zero indicates paternal origin of the allele.  But if we
think carefully about this example, there is symmetry in the pedigree,
namely swapping the two parents, that does not appear in this
partition.  Due to this reason, we need to consider the following
object.

Let $Pr[O|X_t]$ be the probability that the state $X_t$ of the HMM
emits the observed data $O$ at site $t$.  
Then the partition $E$ induced
on the state space by the emission probability is the \emph{emission
partition} containing all distinct sets $E_x$ where
\[
E_x = \{y \in \mathcal{H}_n ~| ~Pr[O=o|X_t=x] = Pr[O=o|X_t=y] ~\forall o \}
\]
and
\[
Pr[O=o| X_t = x] = \sum_{\tilde{o}~consistent~with~R_y} \frac{1}{2^{h(o)}} \prod_{c \in CC(R_x)} Pr[c(\tilde{o})]
\]
where $o$ is a vector of sets, $\tilde{o}$ is a vector of tuples that is an ordered version of $o$, meaning that $o_i \equiv \tilde{o}_i$ while removing the order information from $\tilde{o}_i$, and $c(\tilde{o})$ gives the allele of $\tilde{o}$ that is assigned to that connected component, and $h(o)$ is the number of heterozygous sites in $o$.  Note that each connected component is a tree, and has exactly one founder.
Also, the identity states are consistent with these probabilities, but the identity states are a sub-partition of the emission partition.  Specifically, from our previous example, $0100 \notin D_{1000}$, but $0100 \in E_{1000}$.
Indeed, the emission partition for the example is $\{\{D_{0000}\}, \{D_{1000},D_{0100}\}, \{D_{1100}\}\}$.

We say that the expectation Markov chain $Y_t$ satisfies the
\emph{emission property} if and only if it preserves the emission
partition in order for the corresponding HMM to have the correct
likelihood.  To preserve the emission partition, all the group
elements $T \in G$ must satisfy $T(y) \in E_x$ for all $y \in E_x$ and
for all $x$.

Now, it is necessary to compute the $E_x$ quickly.  The na\"{i}ve algorithm would be slow, since we would have to consider all pairs $x,y$ and all possible data $d$.
Neither can we use the methods in the literature dealing with \emph{condensed identity states}~\cite{Jacquard1972,Thompson1974,Karigl1982}, because the literature computes pedigree-free condensed identity states.  That calculation takes the sets from the identity states and applies permutations of the form $\pi_{i} = (i_m~i_f)$ to swap the alleles of an individual of interest $i \in I$.  However these permutations can violate the inheritance rules specified by a fixed pedigree.  For the example above, take vector $1010 \in D_{0000}$ and swap the alleles of the second child $\pi_2(1010) = 1001 \in D_{1100}$.  This clearly produces a partition that is not the emission partition, and so it would violate the property that we wish to enforce.  Several works on optimal state space reduction for pedigree HMMs have discussed the condensed identity states~\cite{Browning2002,McPeek2002inference} for state-space reduction.  It would appear that they did not formulate the emission partition that was mentioned by Geiger, et al.~\cite{Geiger2009} and that is used here.

The main difference between $D$ and $E$ partitions is that the probability $Pr[D=d| X_t = x]$ has a product over indistinguishable connected components, whereas the identity states distinguishes each connected component.  The partition $D$ must additionally answer the question of which connected components are exchangeable.  Let $I'$ be the individuals of interest having parents who are not individuals of interest.
So, we can rewrite $E_x$ as follows:
\[
E_x = \{y \in \mathcal{H}_n ~|~ \exists \phi ~\textrm{a proper isomorphism s.t.}~ CC(R_x) = CC(\phi(R_y)) \}
\]
where an isomorphism $\phi$ is \emph{proper} if and only if $\phi$ is an isomorphism from $R_y$ to $R_x$ where for all $i\in I' \cup V \setminus I$, either $\phi(i_f) = i_f$ and $\phi(i_m) = i_m$ or $\phi(i_f) = i_m$ and $\phi(i_m) = i_f$.
This definition of $E_x$ is easier to compute, because now we can do an $O(n)$ check to see if the forest of trees in $x$ and $y$ are isomorphic, which leads to an $O(n2^{2n})$ calculation.  However, we can do better.

From the above definition, we see that in order for two inheritance vectors to be isomorphic, the pedigree graph itself (as opposed to the inheritance graph) must have an automorphism.  If we can identify all the relevant automorphisms for the pedigree graph, then we can make a set $A$ of permutations (one for each automorphism), and use a group theoretic algorithm for obtaining the orbits of $\langle A \rangle$ acting on the partition $\{D_x~|~ \forall x \in \mathcal{H}_n\}$ to obtain the desired emission partition.

First to obtain the automorphisms of the graph, we will employ a na\"{i}ve strategy.  Let $i \in I' \cup V \setminus I$ be an individual of interest.
Recall that any proper isomorphism must map one branch of $i$'s ancestral lineage to the other branch.  In order to be consistent, for the set $J = \{i\} \cup \{j ~|~ j \textrm{ full sib of } i\}$, the automorphism must $\phi(j_m) = j_f$ for $j \in J$. 
Considering $i$'s parents and proceeding backward in time, the sub-pedigree connected to the ancestors forms a directed acyclic graph (dag) with in-degree two.  Without loss of generality, we can assume that this sub-pedigree has no individuals in $I \setminus \{i\}$, because, if there were, there would be no proper automorphism and, if there is a descendant of the ancestors not in $I$, it can be trivially removed from the pedigree~\cite{McPeek2002inference}.  Therefore, we may consider only the tree of direct ancestors branching backward in time.
At each branch point, $b$, in this tree, we assign an indicator $\gamma_b = 1$ if the father is to the left and the mother to the right.  There are $O(2^n)$ assignments of these variables $\{\gamma_b~|~\forall b\}$.  For each possible assignment, perform an $O(n)$ graph-traversal operation to check whether this assignment is an automorphism.  We take the first automorphism $\phi$ that we find, because any other $\phi'$ from the same lineage will satisfy $CC(\phi(R_x)) = CC(\phi'(R_x))$ for all inheritance vectors $x$.

Now that we have the automorphisms, we can write them as isometries and put them in set $A$ and consider the orbits of the group $\langle A \rangle$ acting on the identity states.  These orbits are the emission partition.  To obtain these orbits, we will use the well-known orbit algorithm~\cite{Holt2005} from computational group theory which will be recapitulated here.  Notice, that we wish to apply this algorithm to the existing partition $M := \{D_x| \forall x \}$.  Take one set $D_x \in M$ and initialize its orbit  as $O_x := \{D_x\}$.  At the end of the following procedure $O_x$ will contain all the elements in $x$'s orbit.  For every element $D_x \in O$ and every automorphism permutation $a \in A$, compute $y := a(z)~~\forall z \in D_x$.  If $y \notin O_x$, then this $y$ and all the elements in its set $D_y$ are added to $O_x$ and $D_y$ is removed from $M$.  This procedure is repeated until $M$ is empty.  Notice that $CC(y) = CC(a(z))$ is compared to $CC(x)$ to determine if $y$ is is in $O_x$.

Since the comparison $CC(y) = CC(x)$ can be computed in linear time, 
the running time to obtain the automorphisms is $O(n2^n)$ and the
orbit algorithm runs in $O(n2^n)$ time.  This means that obtaining
automorphisms of the pedigree is preferable to checking pairs of
inheritance vectors for isomorphism.

\subsection{Examples}
\label{sec:example}

We will consider two examples, here.  The first is \emph{a} specific three-generation pedigree while the second is a result that applies to \emph{all} two-generation pedigrees.

\subsubsection{Three-Generation Pedigree}

For example, given 4 meioses for two half-cousins, $A$ and $B$, with
one shared grandparent, their common grandparent and their respective
parents who are half-siblings, we have 16 hypercube vertices (see
Figure~\ref{fig:halfcousins}).  Our individuals of interest are $I =
\{A,B\}$.  The emission partition is, in this case, identical to the identity states and contains the sets 
\begin{eqnarray*}
E_1 &=& \{\{A_p\}, \{A_m,B_m\}, \{B_p\} \} \textrm{ and } \\
E_2 &=& \{\{A_p\}, \{A_m\}, \{B_m\}, \{B_p\} \},
\end{eqnarray*} 
since these are the only
partitions of alleles of individuals $I$ that have non-empty sets in
the emission partition.  The emission partition induced on the
hypercube vertices is: $E_{x_1}=\{1001,1111\}$ and
\[
E_{x_2}=\mathcal{H}_n \setminus E_{x_1}.
\]

Notice that in this instance we cannot use the emission partition $\{E_x| ~\forall x\}$ as the state space of a new Markov chain.  For example, if
we were to let $Z_t$ be a Markov chain on the partition given by the emission partition, then the Markov criteria would fail to hold.
Specifically, consider state $x_1=0001$ and $x_2=0011$.  
Then by checking Equation (\ref{eqn:markov}), we have
$\sum_{y \in E_{x_1}} Pr[X_t = y | X_t = 0001] = \theta(1-\theta)^3 + \theta^3(1-\theta)$
but 
$\sum_{y \in E_{x_1}} Pr[X_t = y | X_t = 0011] = 2 \cdot \theta^2(1-\theta)^2$.

The largest partition of ${\mathcal H}_n$ that satisfies the Markov
criteria is 
\begin{eqnarray*}
P_J &=& \{1001,1111\}, \\
P_R &=& \{0010,0100\}, \\
P_G &=& \{1011,1101\}, \\
P_B &=& \{0000,0110\}, \\
P_K &=& \{0011,0101,1010,1100\}, \textrm{ and } \\
P_L &=& \{0001,0111,1000,1110\}.  
\end{eqnarray*}
Let $H$ be the matrix of pair-wise Hamming distances between all the
vertices of the hypercube. 
Then the transition probabilities take the form:
For example, $Pr[Y_{t+1}=P_L | Y_t=P_K] = 2 \theta(1-\theta)^3 + 2 \theta^3(1-\theta)$.

\begin{figure}[ht]
  \begin{center}
    \includegraphics{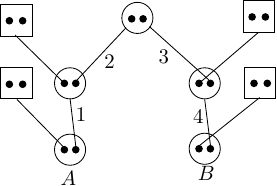}
  \end{center}
  \caption{{\bf Two Half-Cousins.}    \label{fig:halfcousins}
(Left Panel) A pedigree with four non-founders of which two are
half-cousins together with their common grandparent.  As before, the
two black dots for each person represent their two alleles, and the
alleles of each individual are ordered, so that the left allele, or
paternal allele, is inherited from the person's father, while the
right, maternal allele is inherited from the mother.  The two cousins
are labeled $A$ and $B$.  It is easy to see that the only possible IBD
is between alleles $A_m$ and $B_m$, the maternal alleles of
individuals $A$ and $B$, respectively.  (Right Panel) This makes the
four male founders irrelevant to the question of IBD.  The four
meioses are labeled in the order of their bits, left-to-right,and the
inheritance states are represented in binary as $x_1 x_2 x_3 x_4$.
Let $x_i=0$ if that allele was inherited from the parent's paternal
allele, and $x_i=1$ if from the maternal allele.  For instance, $A$
and $B$ are IBD only for inheritance states $1001$ and $1111$.
}
\end{figure}

Notice that this partition can be expressed as the orbits of a group
of isometries, because $G = \langle (1~4), (2~3), \phi_{0110} \rangle$ does not violate the IBD
class.

\subsubsection{Two-Generation Pedigrees}

\begin{lem}
\label{lem:twogen}
For any two-generation pedigree, the partition defined by the emission partition, $C = \{E_x | ~\forall x\}$, satisfies the Markov Property.
\end{lem}
\begin{proof}
We can establish this by finding a group of isometries whose orbits are the emission partition.
This group has the generating set $A$ where $A = \{\phi_f : \forall f\} \cap \{\pi_m: \forall m\}$ and $\phi_f$ and $\pi_m$ are defined as follows.  For founder $f$, $\phi_f$ is a switch having bits set as follows.  Let $i_1,..,i_c$ be the meioses from founder $f$ to each of the founders $c$ children.  Then $\phi_{fi} = 1$ if $i = i_j$ for some $j$ and $\phi_{fi} = 0$ otherwise.  Let $m = (f_1,f_2)$ which are untyped monogamous married founding pairs.  Then $\pi_m = c_1 \circ c_2 \circ ... \circ c_k$ is a permutation composed of $k$ disjoint cycles, one for each child.  For child $i$ with meiosis bits $i_0,i_1$, $c_i = (i_0~ i_1)$.  The group of isometries $G = \langle A \rangle$.

Now, we simply need to establish that the emission partition $C$ is the orbits of this group $G$.  
There is no element $T \in G$ that maps $x \in E_{x_1}$ to $y \in E_{x_2}$, since every $\phi_f$ and $\pi_m$ map the bits of $x$ in ways that maintain $CC(R_x)$.  
Now, we simply need to show that for any $x_1,x_2 \in E_x$, there is always some element $T \in G$ such that $y = T(x)$.  Consider each connected component in $CC(R_x)$ where $x$ and $y$ differ.  The alleles connected in this connected component must all share inheritance through one of the founder bits of the common parents.  If there is only one common parent, the switch for that founder must map between $x$ and $y$ in the bits for that connected component.  If there are two common parents, then there must exist a composition of two founder switches and the founder permutation that maps between $x$ and $y$ for the bits in that connected component.  The complete map $T$ is simply the composition of the isometries for each connected component.
\end{proof}

In the next section, we will introduce the Maximal Ensemble Problem, and we will soon see that this lemma provides a fast method to obtain the optimal partition for two-generation pedigrees.

\subsection{The State-Space Reduction Problem}

There have been three state-space reduction problems posed, we restate
these here.  Given the original pedigree state space $\mathcal{H}_n$,
there are three ways to reduce the state space.
\begin{description}
\item[Maximum Ensemble Problem] Find the 
partition, $\{W_1,...,W_k\}$ of $\mathcal{H}_n$ that satisfies both
the Markov property and the emission property and that minimizes the 
number of sets in the partition: $~\argmax_{\{W_1,...,W_k\}} k$.

\item[Maximum Isometry Group Problem~\cite{Browning2002}] Find the isometry 
group $G$ of maximal size whose orbits $\Omega(G)$ satisfy the emission property. 

\item[Maximum Symmetry Group Problem] Find the symmetry  
group $G$ of maximal size whose orbits $\Omega(G)$ satisfy both the Markov property and the emission property.
\end{description}

We have already proven that all symmetry groups that satisfy the
Markov property have an isometry group with equivalent orbits.  This
means that the later two problems are identical.  Indeed since these
last two problems are equivalent, we will refer to them collectively
as the {\bf Maximum Group Problem.}  The remaining question is the
relationship between the maximum ensemble problem and the maximum
isometry group problem.  We will first introduce a Maximum Ensemble
Algorithm and use it to prove that the solution to the Maximum
Ensemble Problem is unique.  Using the uniqueness result, we will be
able to prove the equivalence of the Maximum Ensemble and Maximum
Isometry Group Problems.

\section{Maximum Ensemble Algorithm}
\label{algorithm}

We will introduce an algorithm that solves the Maximum Ensemble
Problem.  Consider the emission partition containing, $E_x$ for all
$x$ of interest.  Of course the sets in the emission partition are
disjoint.  Consider the $(2^n)!$ permutations on the vertices of the
hypercube.  Naively, these are all candidate permutations for our
group, if we wish to find the maximal group.  However in this section,
we focus on finding the sub-partition of the emission partition that
yields the maximum ensemble solution.  Given the state space, the
partition can be found in linear time.

We do this by iteratively sub-partitioning the partition according to
the coefficients and powers appearing in
Equation~\ref{eqn:markovped}. See Algorithm~\ref{alg:bipartition}:
Bipartition, which takes as input a subpartition of the emission partition.
This recursion is possible since the Markov property must produce a
partition that is a sub-partition of the emission partition (i.e. in order to
respect the emission partition).  Indeed, as shown in
Lemma~\ref{lem:partition}, any pair of vectors $x_1,x_2$ that violate
the Markov property must appear in separate sets of the partition.
This recursive approach will at worst produce a partition with each
element in its own set.

Algorithm~\ref{alg:bipartition} only needs to compute the
$2^n \times 2^n$ matrix of distances between IBD vectors, as well as
do some bookkeeping.  So, the total running time is $O(2^{2n})$.
Since the iterative sub-partitioning at minimum splits sets in two and
does not introduce new inequalities, the number of iterations of the
partition algorithm is $O(log(2^n))= O(n)$.  One iteration of
Algorithm~\ref{alg:bipartition} requires $O(2^{2n})$ time
for each iteration, since we have to check the $2^n \times 2^n$ matrix
of distances between partition elements.  So, the total running time
is $O(n 2^{2n})$.

Now, we need to establish the correctness and uniqueness of the partition.

\begin{algorithm}[ht!]
 \caption{ Bipartition($P$) in $O(2^{2n})$ time }
 \label{alg:bipartition}
 \begin{algorithmic}
 	\REQUIRE ~\\
		$P$: current subpartition of the emission partition \\
	\ENSURE ~\\
		$P'$: violates fewer equations of the Markov property \\
	\MAIN ~\\
	 \STATE $P' = \emptyset$
   	 \FORALL{$W_i \in P$}
	  \STATE $C_{i0} = W_i$ 
	  \STATE $C_{i1} = \emptyset$
    	  \FORALL{$W_j \in P$}
	   \STATE $a_k = 0$ for all $0 \le k \le n$
	   \STATE $s_{x'} = 0$ for all $x' \in C_{i0}$
	   \STATE Let $x_1 \in C_{i0}$ be a fixed element of $C_{i0}$.
	   \FORALL{$x \in C_{i0}$}
 	    \STATE $b_k = 0$ for all $0 \le k \le n$
	    \FORALL {$y \in W_j$}
		\STATE Let $k = |y \oplus x|$
		\IF{$x == x_1$}
			\STATE $a_k++$
		\ENDIF
		\STATE $b_k++$
	    \ENDFOR 
	    \IF{$a_k \ne b_k$ for some $0 \le k \le n$}
		\STATE $s_{x} = 1$
	    \ENDIF
	   \ENDFOR 
	  \STATE \{Bipartition $W_i$\}
	  \FORALL{$x \in C_{i0}$}
		\STATE $C_{i0} \gets C_{i0} \setminus \{x\}$
		\STATE $C_{s_x} \gets C_{s_x} \cup \{x\}$
          \ENDFOR  
         \ENDFOR
	 \STATE $P' \gets P' \cup \{C_{i0},C_{i1}\}$
 	 \ENDFOR
	\STATE RETURN $P'$
  \end{algorithmic}
\end{algorithm}

\begin{lem}
\label{lem:partition}
Let $W_i, W_j$ be two sets of the partition such that 
$x_1,x_2 \in W_i$ and $x_1,x_2$ violate the Markov property in 
Equation~\ref{eqn:markovped}, i.e.~such that
\[\sum_{y\in W_j} s^{|y \oplus x_1|} \ne \sum_{y \in W_j} s^{|y \oplus x_2|}.\]
Then even if $W_j$ is subdivided, $x_1,x_2$ continue to violate
Equation~\ref{eqn:markovped}.
\end{lem}
\begin{proof}
This is proven by a  simple property of polynomials.  Since 
\[\sum_{y\in W_j} s^{|y \oplus x_1|} \ne \sum_{y \in W_j} s^{|y \oplus x_2|},\]
there must be at least one power for which the polynomial coefficients
disagree.  Let $a_k$ and $b_k$ be the coefficients from the left- and
right-had sides respectively.  Let $A(k) = \{y : |y \oplus x_1|=k \}$, so
that $a_k = |A(k)|$, and let $B(k) = \{y : |y \oplus x_2|=k \}$, so that $b_k
= |B(k)|$.  Let $C,D$ be any bipartition of $W_j$.  Therefore $C$ and $D$
induce a partition of $A(k)$ and $B(k)$.  Specifically $A(k)$ is
partitioned into sets $A(k) \cap C$ and $A(k) \cap D$, while $B(k)$ is
partitioned into $B(k) \cap C$ and $B(k) \cap D$.  Since $|A(k)| \ne |B(k)|$, then at least one of
\[|A(k) \cap C| \ne |B(k) \cap C|\]
or
\[|A(k) \cap D| \ne |B(k) \cap D|.\]
Therefore at least one of
\[\sum_{y\in C} s^{|y \oplus x_1|} \ne \sum_{y \in C} s^{|y \oplus x_2|},\]
or
\[\sum_{y\in D} s^{|y \oplus x_1|} \ne \sum_{y \in D} s^{|y \oplus x_2|}.\]
\end{proof}

\begin{lem} (Loop Invariant.)
\label{lem:invariant}
Once $C_{i0}$ is added to $P'$, it is never subdivided again in any
iteration.  This is equivalent to stating the invariant that for any $i$,
\[
\sum_{y \in W_j} s^{|y \oplus x_1|} = \sum_{y \in W_j} s^{|y \oplus x_2|} ~~\forall~ x_1,x_2 \in C_{i0} ~~\forall~ W_j \in P'
\]
\end{lem}
\begin{proof}
Notice that the above invariant is a consequence of both the loop
``foreach $W_j \in P$'' and of the Bipartition algorithm.  For the
base case $C_{i0} = \emptyset ~~\forall i$, and the invariant holds
trivially.

Now we need to inductively prove that the invariant holds.  Assume
that for some $i$, the invariant holds.  Now, consider the loop for a
fixed $W_j \in P$.  $W_j$ may be partitioned into some $C_{j0}$ and
$C_{j1}$.  Our task is to prove that for the new partition of $W_j$ the invariant holds, i.e.~that 
\[
\sum_{y \in C_{j0}} s^{|y \oplus x_1|} = \sum_{y \in C_{j0}} s^{|y \oplus x_2|} ~~\forall~ x_1,x_2 \in C_{i0}.
\]

From the invariant, we have $\sum_{y \in W_j} s^{|y \oplus x_1|} = \sum_{y
\in W_j} s^{|y \oplus x_2|} ~~\forall x_1,x_2 \in C_{i0}$.  Fix $k$ and define the set 
\[A(k,x_1):= \{y \in W_j : |y \oplus x_1| = k\} ~~\forall~ x_1 \in C_{i0},\] then the
coefficient of the $k$th power in the equation is $|A(k,x_1)|$.
Furthermore, we have $|A(k,x_1)|=|A(k,x_2)|$ for all $x_1,x_2 \in
C_{i0}$.

Notice that $C_{j0}$ was created with the property that 
\[\sum_{x_1 \in C_{i0}} s^{|x_1 \oplus y_1|} = \sum_{x_1 \in C_{i0}} s^{|x_1 \oplus y_2|}\]
 for all $y_1,y_2 \in C_{j0}$.  
Define the set
\[B(k,x_1) := \{y_1 \in C_{j0}: |x_1 \oplus y_1| = k\}~~\forall~ x_1 \in C_{i0},\]
and its mirror set
\[D(k,y_1) := \{x_1 \in C_{i0} : |x_1 \oplus y_1|=k\}~~\forall~ y_1 \in C_{j0}.\]
Notice that $A(k, x_1) \cap C_{j0} = B(k,x_1)$ for all $x_1\in C_{i0}$.
 
Now we will use the property $|D(k,y_1)| = |D(k,y_2)|$ for all
$y_1,y_2 \in C_{j0}$ to prove that $|B(k,x_1)| = |B(k,x_2)|$ for all
$x_1,x_2 \in C_{i0}$.  Let $\phi:C_{j0} \to C_{j0}$ be a bijective
map on $C_{j0}$ such that $\phi(x_1)=x_2$. Pick a bijective map $\pi:C_{i0} \to
C_{i0}$ that maps elements of $D(k,y_1)$ to elements of
$D(k,\phi(y_1))$.  Now, we will show that $y_1 \in B(k,x_1)$ if and
only if $\phi(y_1) \in B(k,\pi(x_1))$.  Now $y_1 \in B(k,x_1) = A(k,x_1) \cap C_{j0}$, so this is equivalent to $x_1 \in D(k,y_1)$, which in turn is true if and only if $\pi(x_1) \in D(k,\phi(y_1))$, or if and only if
$\phi(y_1) \in A(k,\pi(x_1))$. Then since $\phi(y_1) \in C_{j0}$, we have that $\phi(y_1)
\in B(k, \pi(x_1))$.

This proves that $|B(k,x_1)| = |B(k,x_2)|$ for all $x_1,x_2 \in C_{i0}$.
Therefore we have
\[
\sum_{y \in C_{j0}} s^{|y \oplus x_1|} = \sum_k |B(k,x_1)| s^{k} ~~\forall~ x_1 \in C_{i0}.
\]
Therefore, we have the invariant that
\[
\sum_{y \in C_{j0}} s^{|y \oplus x_1|} = \sum_{y \in C_{j0}} s^{|y \oplus x_2|} ~~\forall~ x_1,x_2 \in C_{i0}
\]
\end{proof}

\begin{thm} (Uniqueness of the Solution.)
\label{thm:uniqueness}
The Maximum Ensemble Algorithm finds the unique solution to the
Maximum Ensemble Problem.
\end{thm}
\begin{proof}
The partitioning algorithm produces a partition that respects
the emission partition, since it begins with the partition given by the emission partition 
and sub-partitions it.  The algorithm also produces partitions that
respect the Markov property, since it iteratively sub-partitions the
emission partition until the Markov property is satisfied.  Notice
that the algorithm is guaranteed to find such a partition since the
trivial partition, i.e. the original state space, satisfies the Markov
property. Since partition sets are only divided if they violate the Markov
property, the algorithm necessarily finds an optimal partition.  Only
the proof of uniqueness remains.

By Lemma~\ref{lem:partition} the solution is invariant to the order in
which the bipartitions are made, since any $x_1,x_2$ which violate the
Markov property must be put into separate sets of the partition at some point.
Indeed, by Lemma~\ref{lem:invariant} we know that once $C_{i0}$ is
created, it is never partitioned again.  Since we begin with a unique
partition, the emission partition, the sequence of $C_{i0}$, created by 
different calls to Algorithm~\ref{alg:bipartition}, will be the final
sets in the partition, up to reordering.  Therefore the Maximum
Ensemble Algorithm finds the unique partition which is the solution to
the Maximum Ensemble Problem.
\end{proof}

\section{Equivalence}

Now, using the uniqueness of a partition as the solution to the
Maximum Ensemble Problem, we can prove equivalence of the Maximum
Ensemble Problem and the Maximum Isometry Group Problem.

\begin{thm}(Equivalence of Maximum Ensemble Problem and Maximum Isometry Group Problem)
\label{thm:equivalence}
A partition $\{W_1,W_2,...,W_k\}$ is a solution to the Maximum
Ensemble Problem if and only if there is an isometry group $G$ that is
a solution to the Maximum Group Problem having orbits $\Omega(G)$
equivalent to the partition: for all $\omega$, we have $\omega
\in \Omega(G)$ if and only if there exists a set in the partition 
$W_j$ such that $W_j = \omega$.
\end{thm}
\begin{proof}
First, we want to show that if a partition is a solution to the Maximum Ensemble Problem, then there is a group with the equivalent orbits that is a solution to the Maximum Group Problem.
Due to Corollary~\ref{cor:partition}, we know that only
isometry groups satisfy the Markov property.  Any partition which is a
solution for the Maximum Ensemble Problem is also, in particular, the
orbits of a group of isometries, $G$.  Assume that $G$ is not the
maximal isometry group.  Because, if not, then there must be some
isometry which can be added.  And, if it were added, it would join two
orbits into one.  Therefore joining two sets of the partition into
one, which contradicts the assumption that the partition was maximal.
Furthermore, since $G$ satisfies the emission property, its orbits
must be a subpartition of the emission partition.  There is no other group
$G'$ with larger size, since the solution to the Maximum Ensemble
Problem is unique (Theorem~\ref{thm:uniqueness}).  A solution to the Maximum
Ensemble Problem is a solution to the Maximum Group Problem.

For the converse we argue by contrapositive.  That is to say, if $G$ is an group of symmetries and its orbits are not the a solution to the Maximum Ensemble Problem, then the partition given by the orbits of $G$ is not a solution to the Maximum Group Problem.
Assume that partition
$\{W_1,...,W_k\}$ is not a solution to the Maximum Ensemble Problem, 
but that it satisfies Equation~\ref{eqn:markovped} and the emission property.  
Then there must also
exist a maximum ensemble partition $\{V_1,...,V_l\}$ such that $l < k$.  This is because the partition $W$ is not the maximal ensemble partition, and this inequality is strict by the uniqueness proven in 
Theorem~\ref{thm:uniqueness}. 
Because $V$ satisfies the Markov and emission properties, it must be a subpartition of $W$ by Lemma~\ref{lem:partition}.
Therefore, there must exist some $i$, $i',$ and
$j$, such that $W_i \subset V_j$ and $W_{i'} \subset V_j$.

By Corollary~\ref{cor:partition}, there are groups $G^W$ and $G^V$ with
orbits $\{W_1,...,W_k\}$ and $\{V_1,...,V_l\}$, respectively.  Choose
$x_1 \in W_i \cap V_j$ and $x_2 \in W_{i'} \cap V_j$.  Then
$\pi_{x_1,x_2}$ from Theorem~\ref{thm:isometry} will be in $G^V$ and
not in $G^W$.  Therefore, $G^W$ is not a solution to the Maximal Isometry
Group Problem; proving the claim.
\end{proof}

\section{Bootstrapping with Known Isometries}

As noted by Geiger et al.~\cite{Geiger2009}, there are two types of
isometries that can be detected easily.  There are the founder
isometries and the chain isometries where there is an outbred lineage
consisting of multiple ungenotyped generations.  

The founder isometries apply only to ungenotyped founders and are
switches on the bits for the edges adjacent to the founder.
Specifically, if ${i_1},...,{i_c}$ are the meiosis bits between the
ungenotyped founder and each of the $c$ children of the founder, then
the switch is given by the bit vector $X_{i} = 1$ if $i = i_j$ for
some $j$ and $X_{i} = 0$ otherwise.  Since the founder alleles are
indistinguishable (due to the missing genotype), we can fix one bit
adjacent to the founder and enumerate the other bits adjacent that
founder.  These founder isometries can be found in $O(n)$ time.


The chain isometries apply to a lineage of $l$ individuals, from
oldest to youngest $i_1,i_2,...,i_l$ where each individual has exactly
one parent from the lineage, one founder parent, one child, and no
siblings, except $i_l$ which may have any number of siblings.  All
individuals except the most recent must be ungenotyped.  The isometry
is then the permutation on every bit, except the oldest, i.e.~$\pi =
(1_1~i_2~i_3~...~i_l)$ Please see Geiger, et al.~\cite{Geiger2009} and
Browning and Browning~\cite{Browning2002} for examples.  These chain
isometries can be found in $O(n^2)$ time.

It would seem that there are other classes of isometries which can be
found quickly, such as the permutations shown in the example in
Section~\ref{sec:example}.  The exact algorithms for finding other
classes of isometries remain an open problem.  Furthermore, it is
unknown whether all the isometries in the maximal group can be found
efficiently.

\subsection{Representatives}

Let $A$ be a generating set of isometries that generate group $G =
\langle A \rangle$, such as the founder and chain isometries.  In order to compute the bootstrap maximum ensemble states, We need to obtain the orbits of $G$ acting on ${\cal H}_n$.  We can obtain them in $O(k|A|o)$ time where $k$ is the number of orbits and $o = \max_{x\in {\cal H}_n} |\omega(x)|$, provided that orbit membership can be checked in constant time.

Let $M = {\cal H}_n$ initially.  We take any vector $x$ out of $M$ and find its orbit $O$.  Initially let $O = \{x\}$.  Now, for every $x \in O$ and every $a \in A$, compute $y = a(x)$.  If $y \notin O$, add $y$ to $O$ and remove $y$ from $M$.  Repeat until $M$ is empty.

Following this procedure, we have all of the orbits of $G$ acting on ${\cal H}_n$.  For each orbit, we will fix a representative to use in the bootstrap maximal ensemble algorithm.

\subsection{Bootstrap Maximal Ensemble}

Now that we have $k$ representatives, one from each orbit of group $G
= \langle A \rangle$, we can introduce a bootstrap version of the
Maximal Ensemble algorithm.  In this case, we can compute
Equation~(\ref{eqn:markovped}) once per representative.

First, we need to partition our representatives according to the set of the emission partition that they belong to.  Consider the emission partition, $\{E_x | ~\forall x \}$, and partition the representatives into these sets.
Also partition ${\cal H}_n$ according to the emission partition.  These two
equivalent partitions define our initial partitions.

Now, we can recursively sub-divide the representatives whenever
Equation~(\ref{eqn:markovped}) is violated.  Notice that we can
compute this equation with $x$ being the representative and $\omega_j$
is some set of the current partition of ${\cal H}_n$.  Each time we
subdivide the partition of the representatives, we need to also
subdivide the partition of ${\cal H}_n$ in the equivalent fashion.
Suppose that we have representative $x$ that we have put into a new
set in the representative partition.  We obtain the equivalent
partition of ${\cal H}_n$ by creating a new set containing $x$ and all
the vectors $y \in \omega(x)$ the orbit of $x$ under the action of $G$.
The recursive subdivision continues until the Markov property is satisfied.

Since the recursive sub-partitioning at minimum splits sets in two,
the number of iterations required is $O(n)$.  Checking the Markov
properties for each iteration requires $O(k2^{n})$ time where $k$ is
the number of representatives, since we have to check the $k
\times 2^n$ matrix of distances, or sums of distances, between
partition elements.  So, the total running time is $O(n k 2^{n})$.

\section{Running Times}

Notice that the naive calculation of Equation~(\ref{eqn:transitions})
requires $O(k2^n)$ time where $k \le 2^n$ is the number of sets in the
partition and $n$ is the number of meioses in the pedigree.  The
calculation is as follows, for each set $W_i$ in the partition, choose
a representative $x \in W_i$.  For each of the sets in the partitions
$W_j$, compute the transition probability $Pr[X_{t+1} \in W_j | X_t =
x]$.  This last step seems to require enumeration of the inheritance
paths.

The running time of the state-space reduction is the running time of
the ensemble algorithm and the running time of the transition
calculation.  It is interesting to note that calculating the
transition probabilities in Equation~\ref{eqn:transitions} is faster
than the HMM forward-backward algorithm having running time
$O(m2^{2n})$.  This means there is potential to improve the
state-space reduction running time, if there is a more efficient
maximal ensemble algorithm.

Regardless of whether the over-all running time of the state-space
reduction is determined by calculating the transition function or the
ensemble states, all the algorithms here produce savings when the
forward-backward algorithm is run.  This is because a $k$-set partition
of the states results in the forward-backward algorithm having $O(mk^2)$
running time where $m$ is the number of sites.  Furthermore, since the
original state space has an $O(m2^{2n})$ forward-backward algorithm
and the ensemble algorithm is $O(n2^{2n})$, the ensemble algorithm is
more efficient when $n < m$ which is typically the case.  The bootstrap algorithm is even more efficient having a running time of $O(nk2^n)$.

\section{Simulation Results}

We simulated pedigrees under a Wright-Fisher model with monogamy where
each pair of monogamous individuals has a Poisson distributed number of
offspring.  There are $n$ individuals per generation and $\lambda$ is
the mean number of offspring per monogamous pair.  The individuals of
interest, $I$, are the extant individuals, i.e.~those in the most
recent generation or, equivalently, the nodes with out-degree zero.
These pedigrees have no inter-generational mating due to how the
Wright-Fisher model is defined.  To get a half-sibling pedigree, each
edge of the pedigree had 50\% chance of have a new parent drawn at
random.  Since monogamy was not preserved during this random process,
the resulting pedigree had half-siblings.

Running the simulation process and the maximal ensemble algorithm 100
times produced
Figure~\ref{fig:simulation}.  The maximal ensemble algorithm produced 
exponential reductions in the size of the state-space.  Whether the
relationships have half-siblings seems not to influence the practical
applicability of the maximal ensemble algorithm (data not shown).

In practice, the maximal ensemble algorithm seems limited to pedigrees
of roughly 14 meioses while the bootstrap maximal ensemble algorithm
seems limited to about 18 meioses.  Of course, both methods yield the
same reduced state space.  Given the practical success of the
bootstrap maximal ensemble algorithm, we recommend that the bootstrap
maximal ensemble algorithm be employed for state-space reduction.

\begin{figure}[ht]
  \begin{center}
    \includegraphics[scale=0.75]{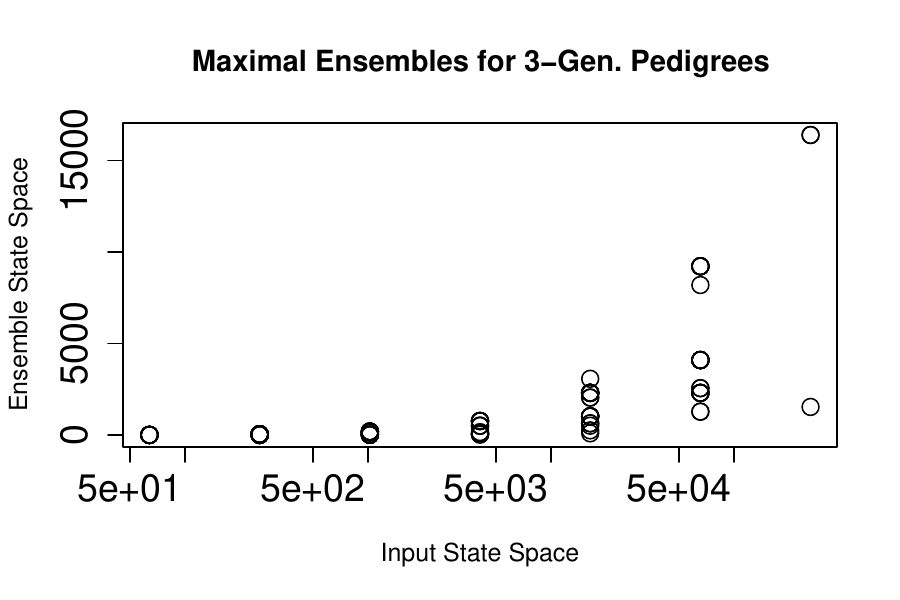}
  \end{center}
  \caption{{\bf Maximal Ensemble Algorithm Results.}   \label{fig:simulation}
The y-axis is the original size of the state space, and the x-axis give the number of ensemble states produced by the maximal ensemble algorithm.  All of the simulated pedigrees had three generations and Poisson mean $\lambda=2$.  One hundred simulation replicates had $n=4$.}
\end{figure}

\section{Discussion}

Even though past efforts at state-space reduction have focused on
finding groups of isometries, it is
clear that this is an equivalent problem to finding the optimal
sub-partition of the emission partition that respects the Markov property.
Although the paper mostly discusses the pedigree state-space, the
maximum ensemble algorithm is general to any HMM.

Even if some isometries can be obtained efficiently, for example the
founder and chain isometries, computation of the transition
probabilities according to Equation~\ref{eqn:transitions} seems to
require enumeration of the inheritance vectors.  The naive algorithm
requires $O(k2^n)$ where $k$ is the number of orbits and $n$ is the
number of meioses in the pedigree.  Due to this fact, and the fact
that the forward-backward algorithm for pedigree HMMs has running time
$O(m2^{2n})$, it is an advantage to use exponential algorithms
to find the maximal state-space reduction.  Indeed, the maximal
ensemble algorithm we introduce here has running time $O(n2^{2n})$
which yields more efficient HMM algorithms when $n < m$ where $n$ is
the number of meioses in the pedigree and $m$ is the number of sites.

In addition to introducing the maximal ensemble algorithm, we
introduced a bootstrap maximal ensemble algorithm which runs in
$O(nk2^{n})$ where $k$ is the number of orbits of the bootstrap
isometry group.  This allows our algorithm to take advantage of known 
isometries such as the founder and chain isometries.

It would appear that there might be an $O(2^{2n})$ algorithm for the
maximum ensemble problem.  This can be seen by the looking at the for
loop of Algorithm~\ref{alg:bipartition}: Bipartition that says
``foreach $x \in A_0$ do''.  This could easily be changed to ``foreach
$A_{\delta}$ and foreach $x \in A_{\delta}$ do''.  However, this
algorithm appears to require sorting the sets in the emission
partition in increasing order by size.  We do not consider the details
of this improved algorithm due to space considerations.

In practice, the maximal ensemble algorithm obtains exponential
reductions in the state-space required for an HMM likelihood
calculation.  The algorithm operates on up to about 18 meioses.

There are several open problems of interest.  First, the computational
complexity of the maximum ensemble problem is open.  Second, an open
problem is the computational complexity of finding the transition
rates after having determined the partition of the state space.
Although naive algorithms are exponential, it is unclear whether
there are approximation algorithms or polynomial-time algorithms for
special cases.

Another very interesting direction is approximation algorithms where
instead of guaranteeing equality in Equation~(\ref{eqn:markovped}), we
could allow for bounded inequalities. Let $Y_t$ be the approximate
Markov chain and $X_t$ be the original Markov chain.  The idea is that
a bound on the inequality for the transition probabilities of $Y_t$
would allow for a larger reduction in the state-space.  In addition,
we would hope that the bound on the inequality would guarantee that
the deviation of $Y_t$'s stationary distribution is bounded relative
to the stationary distribution of $X_t$.

\paragraph{Acknowledgements}
Many thanks go to Yun Song for suggesting the problem and to Eran
Halperin for the random pedigree simulator. K.K. was partially
supported by NSF grants OISE-0730136 and DMS-1106770.

\bibliography{pedigree}
\bibliographystyle{plain}

\end{document}